\documentclass[10pt]{article}
\usepackage{amsmath,amsthm,amsfonts,a4wide}
\usepackage{url}

\newtheorem{theorem}{Theorem}
\newtheorem{lemma}{Lemma}
\newtheorem{remark}{Remark}
\newtheorem{proposition}{Proposition}
\newtheorem{corollary}{Corollary}

\newcommand{\cref}[1]{Corollary~\ref{#1}}

\newcommand{\lref}[1]{Lemma~\ref{#1}}
\newcommand{\pref}[1]{Proposition~\ref{#1}}

\newcommand{\tref}[1]{Theorem~\ref{#1}}
\newcommand{\sref}[1]{Section~\ref{#1}}

\newcommand{\Id}{\mathrm{Id}}
\newcommand{\Isom}{\mathrm{Isom}}
\newcommand{\R}{\mathbb{R}}
\newcommand{\Z}{\mathbb{Z}}
\newcommand{\RP}{\mathbb{RP}}

\title{Geodesic behavior for
   Finsler metrics of constant positive flag curvature on~$S^2$}
\author{ R.~L.~Bryant, P.~Foulon, S.~Ivanov, V.~S.~Matveev, W.~Ziller }

\date{}

\begin{document}

\maketitle

\begin{abstract}

We study non-reversible Finsler metrics with constant flag curvature~$1$ on~$S^2$
and show that the geodesic flow of every such metric is conjugate to that of one of Katok's examples, which form a 1-parameter family.
In particular, the length of the shortest closed geodesic
is a complete invariant of the geodesic flow.
We also show, in any dimension, that the geodesic flow of a Finsler metrics with constant  positive flag curvature is completely integrable.

Finally, we give an example of a Finsler metric on~$S^2$
with positive flag curvature such that no two closed geodesics intersect
and show that this is not possible when the metric is reversible
or have constant flag curvature
\end{abstract}

\section{Introduction}

In Riemannian geometry, metrics of constant sectional curvature
play an important role. In Finsler geometry,
an analog of the sectional curvature is the so-called flag curvature,
which, in dimension~$2$, is a function on the unit tangent bundle
that specializes to the Gauss curvature when the Finsler structure
is Riemannian.

H. Akbar-Zadeh~\cite{Akbar-Zadeh} showed that
if a Finsler metric on a compact surface has constant negative flag curvature,
then it is Riemannian, and, if it has zero flag curvature,
then it is locally Minkowskian.
If a Finsler metric on a compact surface has constant positive flag curvature
and is, in addition, \textit{reversible},
it is a Riemannian metric by~\cite{Bryant2006}.
But, in the non-reversible case,
there are many Finsler metrics on~$S^2$ or~$\RP^2$
with constant positive flag curvature.

The first non-Riemannian examples were constructed in~1973
by A.B.~Katok~\cite{Katok} (see also~\cite{Ziller})
via \textit{Zermelo deformation}
of the round metric on~$S^2$, though it was not realized
at the time that his examples had constant positive flag curvature \cite{Shen2}.

For a Finsler metric~$F$ on an $n$-manifold $M$
and a vector field $X$ on~$M$ such that $F(\pm X(p))\le a_\pm$
for all~$p\in M$, the \textit{Zermelo deformations}~$F_\alpha$
(for $-a_+<\alpha<a_-$) of~$F$ via~$X$
are the $1$-parameter family of Finsler metrics
such that, for each $p\in M$,
the unit sphere of~$F_\alpha$ in~$T_pM$
is the translation of the unit sphere of~$F$
in~$T_pM$ by $\alpha X(p)$.

It is now known~\cite{FM,HuangMo}
(but already observed in~\cite{Foulon}, see also~\cite{BRS,Shen2})
that if $F$ has constant flag curvature~$K$
and $X$ is a Killing vector field of~$F$, then the associated Zermelo deformations~$F_\alpha$
also have constant flag curvature~$K$.
Recall that  Katok's examples are of the form $F_\alpha$,
where $F$ is a Riemannian metric of constant curvature on~$S^2$
and $X$ is one of its Killing vector fields.
Note that Katok's examples are not Riemannian;
they are not even reversible.

The first author constructed several other classes of Finsler surfaces
with constant positive flag curvature~\cite{Bryant1996,Bryant1997,Bryant2002,Bryant2017},
but a full classification is still lacking.

It is known, see~\cite{Shen} and~\sref{prthm1},
that a compact connected Finsler manifold with constant positive flag curvature
is covered by a sphere. In dimension $n=2$, we show that geodesics
of such metrics on the $2$-sphere behave qualitatively
like those of the Katok metrics. Our main result is:

\begin{theorem} \label{geodesicflow}
Let $F_1,F_2$ be two Finsler metrics on~$S^2$ with constant flag curvature~$1$.
Then the geodesic flow of~$F_1$ is conjugate to the geodesic flow of~$F_2$
if and only if the lengths of their shortest closed geodesics are equal.
\end{theorem}

Here `conjugacy' means that there exists a diffeomorphism
of the unit tangent bundles
that takes the geodesic flow of~$F_1$ to that of~$F_2$.

For the proof of \tref{geodesicflow},
the following properties of geodesics  are crucial.

\begin{theorem} \label{geodesic}
Let $F$ be a Finsler metric on $S^2$ with constant flag curvature $1$.
Then there exists an embedded  shortest closed geodesic
of length $2\pi\mu\in(\pi,2\pi]$, and the following holds:
\begin{enumerate}
 \item[(a)] If $\mu=1$, all geodesics are closed and have the same length~$2\pi$.
 \item[(b)] If $\mu$ is irrational, there exist two closed geodesics
with the same image, and all other geodesics are not closed.
The length of the second closed geodesic is $2\pi\mu/(2\mu{-}1)$.
Moreover, the metric admits a Killing vector field.
 \item[(c)] If $\mu=p/q \in(\tfrac12,1)$ with~$p,q \in \mathbb{N}$
and $\gcd(p,q)=1$, then all unit-speed  geodesics have a common period~$2\pi p$.
Furthermore, there exist at most two closed geodesics
with length less than $2\pi p$. A second one exists only if~$2p{-}q>1$,
and its length is~$2\pi p /(2p{-}q)\in (2\pi,2 p \pi) $.
\end{enumerate}
\end{theorem}

We will show that, for each~$\mu\in(\tfrac12,1]$,
there exists a unique Katok metric,
whose shortest closed geodesic has length $2\pi\mu$. Thus:

\begin{corollary} \label{cor:1}
The geodesic flow of a Finsler metric on $S^2$ with flag curvature~$1$
is smoothly conjugate to that of a unique member of the family of Katok metrics.
\end{corollary}

By \tref{geodesic}, we see that the sum of the reciprocals
of the length of the two shortest closed geodesics is always
$1/\pi$. Hence we have:

\begin{corollary} \label{cor:new}
Let $F$ be a Finsler metric on~$S^2$ of constant positive flag curvature~$K$.
Then, the sum of the reciprocals of the length of two shortest closed geodesics
is~$K/\pi$. In particular, geodesic flows of two Finsler metrics on~$S^2$
whose flag curvatures are different positive constants are not conjugate,
even in the continuous sense.
\end{corollary}

Combining \tref{geodesic} and Corollaries~\ref{cor:1}~and~\ref{cor:new}
gives us a complete classification of the geodesic flows of Finsler metrics
on~$S^2$ with constant flag curvature.

At the end of~\sref{sec:GeodFlow},
we will show that \tref{geodesicflow},
\cref{cor:1} and~\cref{cor:new}
also hold for Finsler metrics on~$\RP^2$
with constant flag curvature.
We note that the Katok metrics on~$S^2$
naturally descend to Finsler metrics on~$\RP^2$.

The examples in~\cite{Bryant1996,Bryant1997,Bryant2002} all have~$\mu=1$.
Some of these examples also admit Killing vector fields and hence,
applying a Zermelo deformation to them, one obtains examples with~$\mu<1$.
We do not know any other examples of Finsler metrics on~$S^2$
with constant flag curvature~$1$ and~$\mu<1$.

We will also prove the following results in dimensions $n\ge 2$.

\begin{theorem} \label{entropy}
Let $(S^n, F)$ be a Finsler manifold of dimension~$n\ge2$
of constant flag curvature~$1$.
Then its geodesic flow is Liouville integrable
and has zero topological entropy.
\end{theorem}

\begin{theorem}\label{perdeform}
Let $F$ be a Finsler metric on~$S^n$ with constant flag curvature~$1$
that admits a Killing vector field.
Then there exists an arbitrarily small Zermelo deformation of~$F$
to one with all of its geodesics closed.
Further, when $n=2$, there is a Zermelo deformation
for which all geodesics are closed and of length $2\pi$.
\end{theorem}

In the proof of \tref{geodesic} we will show that,
for metrics of constant positive flag curvature on~$S^2$,
any two geodesics intersect.
(See Propositions~\ref{irrational}~and~\ref{intersect}
in~\sref{prthm1}.)
Our final result concerns the question of what happens
when the flag curvature is not a positive constant.

Recall that, in the Riemannian case, this question
is answered by the famous theorem of Frankel~\cite{Frankel},
which states that,
in an $n$-dimensional Riemannian manifold with positive curvature,
two totally geodesic submanifolds of dimension $n_1$ and $n_2$
must meet if $n_1+n_2\ge n$.

\begin{theorem}  \label{Frankel}
Let $F$ be a Finsler metric on~$S^2$ with positive flag curvature.
Then any two geodesically reversible closed geodesics meet.
Meanwhile, there exists a Finsler metric on~$S^2$
with positive flag curvature and exactly two closed geodesics,
which are disjoint.
\end{theorem}

Here, a geodesic is called {\it geodesically reversible}
if the geodesic with reversed orientation is, up to parametrization,
also a geodesic. In particular, for a reversible Finsler metric
of positive flag curvature on $S^2$, any two closed geodesics intersect.
Our counterexample in the non-reversible case
is a small perturbation of a Katok example.
A similar perturbation was constructed in~\cite{Rademacher}.

We do not know whether an analogue of Frankel's theorem
holds for reversible Finsler metrics in dimensions above~$2$.
In~\cite{Radu} it was claimed that Frankel's theorem holds in all dimensions
for reversible Finsler metrics, but the proof is not correct.
At the end of~\sref{example}, we discuss the difficulties
involved in a straightforward generalization to Finlser metrics
of the proof of Frankel's theorem in the Riemannian case.

Our paper is organized as follows.
In~\sref{prthm1}, we first recall some well-known facts
about Finsler metrics with constant positive flag curvature
and then prove~\tref{geodesic} as well as some other facts
about properties of closed geodesics.
In~\sref{sec:GeodFlow} we prove~\tref{geodesicflow}.
In~\sref{sec:3} we discuss and then prove~\tref{perdeform}.
\tref{entropy} is proved in~\sref{sec:5}
and \tref{Frankel} in~\sref{example}.

\vspace{1ex}
{\bf Acknowledgments.}
The authors thank David Bao, Alexey Bolsinov, Sergei Matveev, Ioan Radu Peter,
Jean-Philippe Pr\'eaux, Hans-Bert Rademacher, Sorin Sabau, Zhongmin Shen,
and Oksana Yakimova for useful comments and help in finding appropriate references.
Some of our results were obtained during the conference
``New Methods in Finsler Geometry'', which took place in July~2016 in Leipzig
and was supported by the DFG and the Universities of Jena and Leipzig.

R.~Bryant thanks Duke University for a research grant
and the U.S. National Science Foundation for the grant DMS-0103884,
S.~Ivanov was supported by the RFBR grant 17-01-00128,
V.~Matveev by the University of Jena and the DFG grant MA 2565/4,
and W.~Ziller by the NSF grant DMS-1506148.

\section{The isometry $\psi$ and properties of closed geodesics} \label{prthm1}

Let $(M^n, F)$ be a compact Finlser manifold of dimension~$n>1$.
For general facts and background about Finsler metrics,
we refer to~\cite{nankai}, see also~\cite{Rademacher2}.

For convenience, we occasionally take advantage
of the \textit{Binet-Legendre} metric~$g_F$,
a Riemannian metric that is naturally associated
to the Finsler metric~$F$~\cite{MT}.
Any isometry of~$F$ is an isometry of~$g_F$,
and hence the group of isometries of~$F$
is a compact Lie group, denoted by~$\Isom(F)$.
(Instead of the Binet-Legendre metric,
one could use a different construction
that naturally associates a Riemannian metric to a Finsler metric,
for example the construction from~\cite{berwald,MRTZ}.)

For $v\in T_pM$ we denote by $\gamma_{p,v}(t)$,
or, for short, $\gamma_{v}(t)$,
the geodesic with $\gamma_{p,v}(0)= p$ and $\dot\gamma_{p,v}(0)= v$,
and call it a \textit{normal} geodesic if $F(p, v)= 1$.
Furthermore, we denote by~$\mathrm{U}(M)$ the unit sphere bundle,
i.e., $\mathrm{U}(M)=\{ v \in TM \mid  F(v)= 1 \}$.

We begin by recalling some well-known facts (see, e.g.,~\cite{Shen})
about Finsler metrics of constant positive flag curvature,
which, by scaling, we can assume to be~$1$.
Let~$F$ be such a Finsler metric on a compact $n$-manifold~$M$.
Since the flag curvature is identically~$1$,
all normal geodesics~$\gamma_{p, v}$
have their first conjugate point at~$t=\pi$ with multiplicity~$n{-}1$.
This implies that for a fixed $p$ the exponential map
sends the sphere of radius~$\pi$ into a fixed point,
which we denote by~$\psi(p)$.
Thus, all geodesics starting at~$p$ meet again at~$\psi(p)$,
i.e., for all normal geodesics~$\gamma_{p,v}(\pi)=\psi(p)$.
The exponential map~$\exp_p$ on the open ball
$B_\pi(p)=\{\xi\in T_pM\mid F(p,\xi)<\pi\}$
is a local diffeomorphism.
By identifying the boundary of~$B_\pi(p)$ to a point
and sending this point to~$\psi(p)$,
the local diffeomorphism extends to a continuous map~$\sigma\colon S^n\to M$.
It is also a local homeomorphism near~$\psi(p)$
since, if two normal geodesics $\gamma_{p,v_1}$ and $\gamma_{p,v_2}$
meet at time close to~$\pi$, they meet again at time $\pi$ in $\psi(p)$,
contradicting the fact that the exponential map
is a local diffeomorphism on all balls of sufficiently small radius.
%
%
Thus $\sigma$ is a covering map.  If $M$ is simply connected,
it follows that it is homeomorphic to a sphere,
that all normal geodesics are minimizing up to time~$\pi$,
and that~$\psi$ has no fixed points.
Thus, for every normal geodesic~$\gamma$,
we have $\psi(\gamma(t))=\gamma(t+\pi)$ for all~$t\in\R$.

We next show that $\psi$ is a smooth isometry
by first showing that it preserves the Finslerian distance~$d_F$.
Indeed, for two arbitrary points $x,y\in M$
take any normal minimizing geodesic~$\gamma$
with $\gamma(0)=x$ and $\gamma(r)=y$.
Then, by the construction of~$\psi$,
we have $\psi(x)= \gamma(\pi)$ and $\psi(y) \gamma(\pi + r)$
and, since $r\le\pi$,
the geodesic $\gamma$ is also minimizing from $\pi$ to $\pi+r$.
Thus, $d_F(x,y)=d_F(\psi(x),\psi(y))$ and, hence, $\psi$ is an isometry,
In particular, $\psi$ is smooth by~\cite{Deng, MT2016}.

Restricting to $n=2$ from now on,
i.e., to a Finsler metric $F$ of flag curvature~$1$ on~$S^2$,
we claim that $\psi$ is an orientation-reversing isometry.
Indeed, the orientation induced by the velocity vector
of a geodesic and a Jacobi vector field along this geodesic
changes sign when we pass the point where the Jacobi vector field vanishes.

Now, in the non-Riemannian case, $\psi^2$ need not be the identity.
In fact, if $\psi^2=\Id$ does hold,
then, for any normal geodesic~$\gamma$,
we have $\gamma(t)=\psi^{2}\bigl(\gamma(t)\bigr)=\gamma(t+2\pi)$ for all~$t$,
and hence~$\gamma$ is periodic with period~$2\pi$.
Since there are no geodesics loops of length~$\pi$,
each geodesic is embedded and has length~$2\pi$.
We will set aside this special case in the following discussion,
but note that there are many non-Riemannian examples
for which $\psi^2=\Id$, see~\cite{Bryant1996,Bryant1997,Bryant2002}.

We now recall the well-known fact that, for every compact Lie group~$G$
acting effectively on~$S^2$, its action is conjugate to that of a subgroup
of~$\mathrm{O}(3)$ acting linearly on~$S^2$.
In particular, there exists a $G$-invariant Riemannian metric
of constant Gauss curvature~$1$ on~$S^2$.
(This seems to be a folklore result,
see e.g.,~\cite{Schultz} or~\cite[Theorem 12]{BMF} for connected groups.
For completeness, we recall its proof:
Take any Riemannian metric on~$S^2$
that is invariant with respect to~$G$.
By the uniformization theorem, it is conformal
to a constant Gauss curvature metric on~$S^2$
and is therefore a compact subgroup in the group of the conformal transformations
of such a metric on~$S^2$.  Since any maximal compact subgroup
of the group of conformal transformations is conjugate to~$\mathrm{O}(3)$,
the claim follows.)

Applying this result to the group of isometries
of the Finsler metric~$F$ on~$S^2$, which, as noted, is compact
and contains~$\psi$, it follows that, in particular,
there is a constant curvature~$1$ Riemannian metric~$g$ on~$S^2$
with respect to which~$\psi$ is an isometry,
and hence the orientation-preserving $g$-isometry~$\psi^2$
(which, by hypothesis, is not the identity)
has exactly two fixed points, which are exchanged by~$\psi$.
Once such a $g$ and and a fixed point~$p$ of~$\psi^2$ have been chosen,
there exist $p$-centered, $g$-spherical polar coordinates~$(\theta,\phi)$
on~$S^2$ with $-\tfrac12\pi\le\theta\le\tfrac12\pi$
and~$\phi\in\R/(2\pi\Z)$,
and a constant~$\lambda$ such that
\begin{equation} \label{eq:psi}
\psi(\theta,\phi)=(-\theta,\,\phi + 2\pi\lambda),
\end{equation}
where $p$ corresponds to~$\theta=-\tfrac12\pi$.

\begin{remark}  \label{rem:1}
For use in \pref{prop:4} below,
we note that the choices of~$g$ and~$p$
determine the coordinates~$(\theta,\psi)$
satisfying~\eqref{eq:psi} uniquely up to the evident ambiguity
$(\theta, \phi)_{\textrm{new}}= (\theta, \pm\phi + \textrm{const})$.
Replacing~$p$ by~$\psi(p)$ replaces $\theta$ by $-\theta$.

In the case $\psi^2=\Id$, there is no restriction on the
choice of~$p$.  However, in this case, under the additional assumption
that there exists precisely one nonzero Killing vector field
up to constant multiples,
one can further specify either of the two zeros
of this Killing field to be~$p$,
and this is what what we will usually do.
$($Among Finsler metrics on $S^2$,
only a Riemannian metric of constant Gauss curvature
can have more than one nonzero Killing field up to constant multiples anyway.$)$
\end{remark}

We can assume that $\lambda\in(0,\tfrac12]$,
since changing the orientation of~$\phi$ replaces~$\lambda$ by~$1{-}\lambda$.
Notice that $\lambda$ cannot equal~$0$, since $\psi$ has no fixed points.
Similarly, $\lambda=\tfrac12$ would imply~$\psi^2=\Id$;
thus, we have~$\lambda<\tfrac12$.
We will call~$\lambda\in(0,\tfrac12)$ the \textit{rotation angle} of~$F$
and sometimes use the notation~$F_\lambda$ when~$F$
is a Finsler metric with rotation angle~$\lambda$.


Simultaneously replacing the `pole' $p$ by~$\psi(p)$
and reversing the orientation determined by~$\phi$
does not change the value of $\lambda\in (0,\tfrac12)$.
Hence this value of~$\lambda$ is an invariant of~$F$.
(Since the Katok examples on~$S^2$ are invariant
under the antipodal map of~$S^2$, which switches the poles
and reverses the orientation, it follows that one cannot always
resolve the paired ambiguity of choice of pole and orientation.)


We distinguish two cases:
When $\lambda$ is rational and when $\lambda$ is irrational.
We start with the latter case.

\begin{proposition}\label{irrational}
Let $F=F_\lambda$ be a Finsler metric on $S^2$ with $\lambda$ irrational.
Then
\begin{enumerate}
 \item[$($a$)$] $(\theta,\phi)\mapsto (\theta,\phi+s)$ is an isometry for all~$s$,
      i.e., $\frac{\partial }{\partial \phi}$ is a Killing vector field,
 \item[$($b$)$] The equator $\{\theta=0\}$ is a geodesic
 that has length $\pi/\lambda>2\pi$ when traversed in the positive direction
 and $\pi/(1{-}\lambda)\in(\pi,2\pi)$ when traversed in the opposite direction.
 \item[$($c$)$] For any other geodesic there exists a number~$\theta_0>0$
 such that the geodesic oscillates between the parallels
 $\theta=\theta_0$ and $\theta=-\theta_0$. Furthermore, the $\theta$-coordinate
 along the geodesic changes monotonically in between these parallels,
 while the $\phi$ coordinate increases by $2\pi\lambda$ or $2\pi(1{-}\lambda)$.
 In particular, the geodesic is not closed.
\end{enumerate}
\end{proposition}

In part (c) we allow $\theta_0=\pm\tfrac12\pi$
for the geodesics going through the poles.

\begin{proof}
Since $\lambda$ is irrational, the set of the points
$\{4k \lambda \pi\}_{k\in \Z}$
is dense in the circle $\R/(2\pi\Z)$.
Thus the closure of the the subgroup of~$\Isom(F)$
consisting of the powers of~$\psi$ contains all transformations
of the form~$(\theta,\phi)\mapsto (\theta,\phi+s)$, which implies (a).
Notice also that
$\Isom(F)\supseteq \{(\theta, \phi) \mapsto (\pm\theta, \pm \phi + s)\}$
and that any isometry preserves or switches the poles.

In order to prove (b), observe that composing~$\psi$
with the flow of the Killing vector field,
we have that $(\theta,\phi)\to(-\theta,\phi)$ is an isometry
and thus its fixed point set $\{\theta=0\}$ is a closed geodesic of~$F$
in either direction, but with possibly different parameterizations.
Let $\gamma^\pm(t)=(0,\pm t)$ be the equator.
The Killing vector field $\frac{\partial }{\partial \phi}$
is tangent to $\gamma^\pm$ and hence they are parametrized
proportional to arc length. For $\gamma^+(t)=(0,t)$,
the length of the curve, for $0\le t \le 2\pi \lambda$,
is equal to~$\pi$ by definition of $\psi$.
Thus $F(\dot\gamma^+)=\tfrac12\lambda$
and hence the length of $\gamma^+$ is $\pi/\lambda$.
Similarly, the length of $\gamma^-$ is $\pi/(1{-}\lambda)$.

Let us now prove (c).
Observe that the Killing vector field $X=\frac{\partial }{\partial \phi}$
is a Jacobi field along any normal geodesic~$\gamma$.
Since the flag curvature is positive and constant,
this implies that for any $\gamma\ne\gamma^\pm$,
there exists a time~$t_0$ such that~$X$ is proportional to~$\dot\gamma$
precisely when $t=t_0+k\pi$ for some $k\in\Z$.
Thus along $\gamma$ the function $\theta$ reaches its maximum and minimum
at $\theta=\pm\theta_0$ and $\theta$ is  monotonic in between.
The geodesic $\gamma$ bounces back and forth between these parallels,
increasing its $\phi$ coordinate by $4\lambda\pi$ or $4(1{-}\lambda)\pi$
(depending on its orientation) each time it returns.
To see that it does not increase by $4\lambda\pi+2\pi k$
or $ 4(1{-}\lambda)\pi+2\pi k$ for some $k\ne 0$,
we observe that this property would be the same for all geodesics,
but does not hold for the geodesics on the equator.
\end{proof}

For fixed $\theta_0>0$, the geodesics from \pref{irrational}(c)
viewed as curves on $U(S^2)$ form two 2-tori in $U(S^2)$,
depending on whether they rotate clockwise or counter clockwise.
The geodesics through the poles form a 2-torus as well.
The geodesic flow has constant rotation number $\lambda$
on all of these 2-tori and degenerates for $\theta_0\to 0$
to the two closed geodesics~$\gamma^\pm$.
\smallskip

Before stating the next result, let us make the following definition:
Given a curve $\alpha\colon [a,b]\to S^2$ not going through either pole,
we define the change of the angular coordinate~$\phi$ as follows:
Assume, as we may, that, in our coordinates $(\theta,\phi)$,
we have $\phi\bigl(\alpha(t)\bigr)) \equiv \hat\phi(t) \mod 2\pi$
for some continuous function $\hat\phi:[a,b]\to\R$.
We call $\hat\phi(b){-}\hat\phi(a)$ the \textit{change} of~$\phi$ along~$\alpha$,
and we define the \textit{winding number} (around the `south' pole, $\psi(p)$)
of $\alpha$ to be $\bigl(\hat\phi(b){-}\hat\phi(a)\bigr)/(2\pi)$.
The winding number of~$\alpha$ is independent
of the choice of~$\hat\phi$ and is an integer when $\alpha(b)=\alpha(a)$.

If $\lambda$ is rational, $\psi$ has some finite order,
which, since $\psi$ reverses orientation, must be even.
Thus $\psi^{2q}=\Id$ for some minimal $q\in \mathbb{N}$.
We set $\lambda=p/(2q)$ with $\gcd(p,q)=1$.
Since we are assuming $\lambda<\tfrac12$, we have $1\le p < q$.
For any normal geodesic, we have $\gamma(t)=\psi^{2q}(\gamma(t))
=\gamma(t+2q\pi)$ for all $t$,
and hence all normal geodesics are periodic with period $2q\pi$.
We call a closed geodesic \textit{exceptional}
if its length is less than $2q\pi$.
We start with the following observation.

\begin{proposition}\label{intersect}
Let $F_\lambda$ be a Finsler metric on~$S^2$ with constant flag curvature~$1$.
Then any two closed geodesics intersect.  If $\lambda=p/(2q)<\tfrac12$
in lowest terms, and one of the geodesics does not
pass through the poles, then they intersect in at least $q>1$ points.
\end{proposition}

\begin{proof}
If $\lambda$ is not rational, then, by~\pref{irrational},
there are only two closed geodesics, and they have the same image in~$S^2$,
so the claim follows.

Now suppose that $\lambda = p/(2q)$ in lowest terms.
If two geodesics pass through the poles, then they intersect.
Suppose that~$\gamma$ is a geodesic that does not pass through the poles.
By the formula for~$\psi$ in local coordinates,
the number of points in an orbit $\{\psi^k(x),\ k\in\Z\}$
other than the poles is equal to~$q$ or~$2q$.
Thus, since geodesics are invariant under~$\psi$,
to prove that any other geodesic intersects $\gamma$ in at least~$q$ points,
it suffices to show that it meets $\gamma$ somewhere.
Now, the complement of~$\gamma$ in~$S^2$
is a countable disjoint union of open sets~$U_i$, each homeomorphic to a disc,
which $\psi$ permutes.  If a second closed geodesic did not meet~$\gamma$,
it would lie in one of these components, say~$U_k$.
But then $\psi(U_k)=U_k$, and hence~$\psi$ would have a fixed point,
which it does not.
\end{proof}

We now determine the length of all exceptional geodesics.

\begin{proposition}\label{rational}
Let $F_\lambda$ be a Finsler metric on $S^2$
with rotation angle $\lambda=p/2q$, $0<p<q$ and $\gcd(p,q)=1$.
Then there are at most two exceptional geodesics, they are embedded,
and we have:
\begin{enumerate}
 \item[$($a$)$] There exists a closed geodesic
 of length $\pi/(1-\lambda)=2q\pi/(2q-p)\in(\pi,2\pi)$,
 that has winding number $-1$ around the south pole.
\item[$($b$)$] There exists a second exceptional closed geodesic
 if and only if $p>1$. Its length is $\pi/\lambda= 2q\pi/p\in(2\pi,2q\pi)$
 and it has winding number $+1$ around the south pole.
\item[$($c$)$] If $p$ is even, the two exceptional geodesics
 have images on the equator.
\end{enumerate}
\end{proposition}
\begin{proof}

Since all geodesics are closed with period $2q\pi$,
the geodesic flow generates a Seifert fibration on~$U(S^2)\simeq \RP^3$.
It is well known (see, e.g., \cite[\S\S 5.3, 5.4]{orlic})
that any Seifert fibration on a lens space has at most two singular leaves.
Thus all geodesics, except possibly two, have the same length $2q\pi$.
We first show that the winding number of an exceptional closed geodesics
is well defined:

\begin{lemma}\label{embedded}
 Every exceptional closed geodesic $\gamma$ is embedded.
\end{lemma}

\begin{proof}
If the length of a geodesic $\gamma$ is not a multiple of $\pi$,
then it equals $k\pi+r$ for some $0<r<\pi$.
Since $\psi$ acts on~$\gamma$ by adding $\pi$
to the parameter, $\psi^{k+1}$ acts by adding $\pi-r$,
which is positive and less than $\pi$.
Hence every arc of $\gamma$ of length $\pi-r$
is the unique minimizing geodesic from a point
to its image by $\psi^{k+1}$.
If $\gamma$ has a self intersection at a point $p_0$,
then there would exist two minimizing geodesics from $p_0$
to~$\psi^{k+1}(p_0)$ which contradicts $\pi-r<\pi$.
If the length is~$k\pi$ with $k<2q$, then $\psi^k$
is the identity on $\gamma$ and hence $\psi^k$
has infinitely many fixed points.
Since~$\psi^k\ne\Id$,
this can only happen if $\psi^k(\theta,\phi)=(-\theta,\phi)$.
It follows that, in this case, $\gamma$ lies on the equator
and is hence embedded.
\end{proof}

Notice that the geodesics through the poles are not embedded,
and hence every exceptional geodesic $\gamma$
has a well defined winding number $W$, and $W=1,-1$, or $0$
since it is embedded.

\begin{lemma}\label{length}
Let $\gamma$ be an embedded closed geodesic. Then one of the following holds:
\begin{enumerate}
 \item[$($a$)$] $\gamma$ has length $\pi/\lambda$
  and its winding number is $+1$,
 \item[$($b$)$] $\gamma$ has length $\pi/(1{-}\lambda)$
  and its winding number is $-1$.
\end{enumerate}
\end{lemma}

\begin{proof}
Recall that every geodesic is $\psi$-invariant
since $\psi(\gamma(t))=\gamma(t+\pi)$. Fix a point~$x_0$ on $\gamma$.
Let $x_0,x_1,...,x_n=x_0$ be the points of the $\psi$-orbit of~$x_0$,
enumerated in the order of appearance on~$\gamma$.
Note that $n>2$, since the case $\psi^2=\Id$ is excluded.

Let $\gamma_i$ denote the arc of~$\gamma$ from~$x_i$ to~$x_{i+1}$.
Each $\gamma_i$ is the image of $\gamma_0$
under some power of $\psi$, so they all have the same length
and the same change of angular coordinate~$\phi$.
Hence the change of angular coordinate along~$\gamma_0$,
let us call it $A$, equals $2\pi W/n$.

One of the points~$x_i$ is $\psi(x_0)$.
Let it be $x_m$ with $1\le m\le n-1$.
The angular coordinate change from $x_0$ to $x_m$ equals $mA=2\pi Wm/n$,
which is a number between $-2\pi$ and $2\pi$ since $|W|\le 1$.
Also this number is equivalent to $2\pi\lambda$ modulo $2\pi$.
Hence it is either $2\pi\lambda$ or $2\pi(\lambda{-}1)$.
If it is $2\pi\lambda$, then $Wm/n=\lambda$,
hence $W=1$ and $m/n=\lambda$.
If it is $2\pi(\lambda{-}1)$, then $Wm/n=\lambda{-}1$,
hence $W=-1$ and $m/n=1-\lambda$.

It remains to determine the length of $\gamma$.
The length from $x_0$ to $x_m$ is $\pi$, since $x_m=\psi(x_0)$.
Thus the length of $\gamma_0$ is $\pi/m$,
and hence the length of $\gamma$ is $n\pi/m$.
Substituting the above formulas for $m/n$ in terms of $\lambda$,
one obtains that the length of $\gamma$ is $\pi/\lambda$
(resp., $\pi/(1{-}\lambda)$).
\end{proof}

We now prove the existence and uniqueness
of the closed geodesics in \pref{rational}.

First observe that if $p$ is even, $q$ must be odd,
and hence $\psi^q(\theta,\phi)=(-\theta,\phi+p/2 \cdot 2\pi)=(-\theta,\phi)$.
Thus the fixed point set of $\psi^q$,
i.e., the equator, is a geodesic in both directions.
Since they are embedded, \lref{length}
implies that their lengths are~$\pi/(1{-}\lambda)$ and~$\pi/\lambda$.
Since $p>1$, both lengths are less than $2q\pi$
and hence both closed geodesics are exceptional.
This finishes~\pref{rational} when $p$ is even.
From now on we can thus assume that $p$ is odd.

Recall that if $f$ is an isometry without fixed points and of finite order,
then we obtain an $f$-invariant closed geodesic
as the minimum of the displacement function~$p\to d_F\bigl(p,f(p)\bigr)$:
Let $D$ be this minimum and let~$s$ be a minimal geodesic from~$p$ to~$f(p)$.
Then $s$ and $f(s)$ form a geodesic, i.e., there is no angle at~$f(p)$,
since otherwise the distance between the midpoints would be less than $D$.
Thus $s, f(s), f^2(s),\ldots$ form a geodesic which eventually closes up,
and $f(\gamma(t))=\gamma(t+D)$.
We call $\gamma$ a \textit{minimum displacement geodesic} of~$f$.

We apply this construction to $f=\psi^k$,
where $k$ is odd and $k\not\equiv 1 \bmod 2q$.
In this case we have $D<\pi$ since~$\psi$
is the only isometry with minimal displacement~$\pi$.
Furthermore, $f$ has no fixed points since $k$ and $p$ are odd.
This implies that a minimum displacement geodesic of~$\psi^k$
exists and is unique. Indeed, by \pref{intersect},
any two such geodesics~$\gamma_1,\gamma_2$ intersect:
$\gamma_1(t_1)=\gamma_2(t_2)$ for some $t_1,t_2$.
Then $\gamma_1(t_1+D)=\gamma_2(t_2+D)$,
which means there are two minimizing geodesics from~$\gamma_1(t_1)$
to~$\gamma_1(t_1+D)$ with $D<\pi$, a contradiction.
The same argument applies to self-intersections,
proving that the minimum displacement geodesic of~$\psi^k$ is embedded.

We can use this to show that there exists
an exceptional closed geodesic with length $<2\pi$.
Let $\gamma_0$ be the minimum displacement geodesic
for $\psi^{-1}$ and $D$ the minimum displacement value.
Then $\gamma_0(D)=\psi^{-1}(\gamma_0(0))$
and hence $\gamma_0(D+\pi)=\psi(\gamma_0(D))=\gamma_0(0)$.
Thus $\gamma_0$ closes up with length $\pi+D<2\pi$
and is hence exceptional.  We furthermore observe that~$\gamma_0$
is the only closed geodesic of length $<2\pi$.
Indeed, suppose that $\gamma_1$
is another closed geodesic with length $<2\pi$.
Then by \lref{length} both $\gamma_0$ and $\gamma_1$
have length~$\pi/(1{-}\lambda)$, which is then equal to $\pi+D$.
Hence the arc of~$\gamma_1$ from $\gamma_1(\pi)$
to $\gamma_1(0)=\psi^{-1}(\gamma_1(\pi))$ has length~$D$.
Therefore $\gamma_1$ is also a minimum displacement geodesic for~$\psi^{-1}$
and thus agrees with $\gamma_0$.
This proves part (a) in \pref{rational}.

We now prove the existence of the second exceptional geodesic when $p>1$.
Since $p$ is odd, there exists an odd~$k$ such that $kp\equiv 1\bmod 2q$
and $k\not\equiv 1\bmod 2q$.
As shown above, there is a minimum displacement geodesic~$\gamma_1$
for $\psi^k$, and $\gamma_1$ is embedded.
Since $p>1$, we have $\max\{\pi/\lambda,\pi/(1{-}\lambda)\}<2q\pi$
and hence~\lref{length} implies that $\gamma_1$ is exceptional.
We still need to show that $\gamma_1$ is different
from the first exceptional geodesic $\gamma_0$.
So assume that $\gamma_1=\gamma_0$.
Let $L=\pi/(1{-}\lambda)=2q\pi/(2q-p)$ be the length of $\gamma_0$
and let $\alpha$ denote the arc of $\gamma_0$ from $x_0=\gamma_0(0)$
to $\psi^k(x_0)=\gamma_0(k\pi)$.
Since $(pk-1)/ 2q\in\Z$, it follows that
$$
\left(k\pi- \frac{(2q-1)\pi}{2q-p}\right)/L = \frac{(2q-p)k-(2q-1)}{ 2q}\in\Z
$$
Furthermore $\displaystyle0<\frac{(2q{-}1)\pi}{2q{-}p}<L$
and hence the length of~$\alpha$
equals $\displaystyle\frac{(2q{-}1)\pi}{2q{-}p}>\pi$.
But then $\alpha$ is not minimizing and therefore $\gamma_0$
is not a minimum displacement geodesic for $\psi^k$.
This proves that $\gamma_1$ is the second exceptional geodesic
and by~\lref{length} has length $\pi/\lambda$.

If $p=1$, we have $\pi/\lambda=2q\pi$
and hence an exceptional geodesic has length $\pi/(1{-}\lambda)<2\pi$.
But as we saw above, there can be only one such geodesic.
This finishes the proof of~\pref{rational}.
\end{proof}

Combining Propositions~\ref{irrational} and~\ref{rational}
finishes the proof of \tref{geodesic} in the Introduction.
Notice though that $\displaystyle\mu=\frac{1}{2(1-\lambda)}$.

We point out that the shortest closed geodesic~$\gamma$
has the following geometric interpretation:
As follows from the proof of~\tref{geodesic},
$\gamma$ is a minimum displacement geodesic for $\psi^{-1}$
and, since $d(p,\psi^{-1}(p))=d(\psi(p),p)$,
it can be viewed as realizing the shortest return time
from $\psi(p)$ to~$p$.

\section{The proof of \tref{geodesicflow} } \label{sec:GeodFlow}

We now prove the conjugacy of the geodesic flow
of two Finsler metrics $F_\lambda$
with the same rotation angle~$\lambda$.

We will first show that the geodesic flow $g_t$,
regarded as an action of $\R$ on $\mathrm{U}(S^2)\simeq\RP^3$
is conjugate to an orthogonal action, i.e.,
a homomorphism $\R\to \mathrm{SO}(4)/\{\pm\Id\}$
where $\mathrm{SO}(4)$, acting orthogonally on $S^3$,
descends to an action on $\RP^3$.

If $\lambda=p/2q$ is rational,
then all geodesics are closed with common minimal period~$2q\pi$
and hence $g_t$ induces an effective action of $S^1$ on $\RP^3$.
But a circle action on $\RP^3$ is conjugate to an orthogonal action,
see~\cite{Raymond}, which can also be regarded
as a homomorphism $\R\to \mathrm{SO}(4)/\{\pm\Id\}$.

If $\lambda$ is irrational, we saw in \pref{irrational}
that we also have a Killing vector field $X=\frac{\partial }{\partial \phi}$
with flow $\tau_t$, which induces a vector field $X^*$ on $\mathrm{U}(S^2)$
with flow $\tau^*_t=d(\tau_t)$ and $\tau^*_{2\pi}=\Id$.
Clearly, $g_t$ and $\tau^*_s$ commute, since $\tau_t$
acts by isometries and hence $[X^*,G]=0$,
where $G$ is the geodesic flow vector field.
From the definition of $\psi$, it also follows
that $g_{2\pi}=\tau^*_{4\lambda\pi}$.
Thus, the flow of the vector field $Y=G-2\lambda X$
is $2\pi$ periodic. Hence $X$ and $Y$ commute
and generate an action of $T^2$ on $\mathrm{U}(S^2)$.
In~\cite{Ne} it was shown that, up to an automorphism of~$T^2$
and a diffeomorphism of $\RP^3$,
there exists a unique effective $T^2$ action on $\RP^3$
which can hence be thought of as an  orthogonal action,
i.e., a homomorphism $T^2\to \mathrm{SO}(4)/\{\pm\Id\}$.
Hence the same is true for the flow $g_t$ of $G$.

Thus in both cases, the geodesic flow is conjugate
to a homomorphism $\R\to \mathrm{SO}(4)/\{\pm\Id\}$.
Such a homomorphism is furthermore conjugate,
via an element of $O(4)$, to a block diagonal one of the form
$t\mapsto \mathrm{diag}(R(at),R(b\,t))/\{\pm\Id\}$,
with $a,b\in\R$, $a\ge b\ge 0$
and where $R(t)$ is a rotation of $\R^2$ with angle~$t$.
Furthermore, $b>0$ since the geodesic flow is almost free.
But such a homomorphism is uniquely characterized
by its smallest minimal periods, i.e.,
$\pi/a$ and $\pi/b$, corresponding to the orbits
through the images of $(1,0,0,0)$ and $(0,0,0,1)\in S^3$ in~$\RP^3$.
As we saw in~\pref{irrational} and~\pref{rational},
the minimal periods of the geodesic flow are $\pi/\lambda$
and $\pi/(1{-}\lambda)$. This implies that $\lambda$
uniquely determines the geodesic flow up to conjugacy.
Since $\pi/(1{-}\lambda)$ is also the length of the shortest closed geodesic,
this finishes the proof of \tref{geodesicflow}.

We now indicate how to show that  the analogs of our results
also hold for Finsler metrics on $\RP^2$ with constant flag curvature.
First we can apply the previous results to the lift of the metric to~$S^2$.
In particular this defines the geometric parameter~$\lambda$.
Then, by~\cite{Raymond} and~\cite{Ne}, the geodesic flow
is again conjugate to an orthogonal action
since the unit tangent bundle of~$\RP^2$
is the lens space $L(4,1)=S^3/\Z_4$ with $\Z_4$
generated by $\textrm{diag}(R(\frac{\pi}{2}),R(\frac{\pi}{2}))$.
The geodesic flow is thus conjugate to
$t\to \textrm{diag}(R(at),R(b\,t))/\Z_4$, with $a,b\in\R$.
Note that the parameters $a$ and $b$ are the same,
up to a sign, as the similar parameters
for the covering metric on~$S^2$
constructed in the proof of \tref{geodesicflow}.
Indeed, there exists only one two-sheeted covering of~$L(4,1)$;
it is the natural projection from~$\RP^3=S^3/\Z_2$
to~$L(4,1)=S^3/\Z_4$.
Hence, $|a|+|b|=1$, and the two minimal periods on $L(4,1)$
are~$\pi/(2|a|)$ and $\pi/(2|b|)$.
In view of orthogonal conjugations in~$L(4,1)$,
we may assume that $a\ge |b|> 0$.
However, unlike the case of~$\RP^3$,
orthogonal conjugations cannot change the signs of $a$ and $b$ individually.

To distinguish the actions with invariants $(a,b)$ and $(a,-b)$,
we compute the return map, i.e.,
the derivative of the flow orthogonal to the orbit,
for the shortest periodic orbit $t\mapsto \mathrm{diag}(R(at),0)$.
The image of the orbit in $L(4,1)$ closes at time $t_0=\pi/(2a)$
and hence the return map
is given by~$R\bigl( b\,\pi/(2a)\bigr)\circ R(\tfrac12\pi)^{-1}$.
It is hence a rotation whose angle,
up to a choice of orientation on~$L(4,1)$,
lies in $[0,\tfrac12\pi )$ if $b>0$
and in  $(\tfrac12\pi ,\pi ]$  if $b<0$.
In particular, the action with $b>0$
is not conjugate to the one with $b<0$.

For a Finsler metric on~$\RP^2$ with flag curvature~$1$,
consider the closed geodesic~$\gamma$
corresponding to the above shortest orbit.
It has length~$\pi/(2a)=\pi/\bigl(2(1{-}\lambda)\bigr)$.
The return map of the geodesic flow along~$\gamma$
can be computed in terms of Jacobi fields
in a parallel basis. The deck transformation
of the two-fold cover is orientation-reversing,
and it preserves the lift of~$\gamma$.
Hence, the parallel translation along~$\gamma$
is a rotation by~$\pi$ in~$T_{\dot\gamma(0)}\mathrm{U}(\RP^2)$.
Looking at Jacobi fields, one sees
that the rotation angle of the return map
differs from that of parallel translation
by~$\pm\textrm{length}(\gamma)$.
Hence, the rotation angle
(normalized to the interval~$[0,\pi]$)
equals~$\pi-\pi/\bigl(2(1{-}\lambda)\bigr)
\in [0,\frac{\pi}{2})$. Therefore,~$b>0$.

Thus, as before, the geodesic flow of the Finsler metric
is determined by~$\lambda$ and, hence, also by the length
of the shortest closed geodesic.

Altogether, we see that the statements in
\tref{geodesicflow}, \cref{cor:1}, and ~\cref{cor:new}
also hold for Finsler metrics on $\RP^2$ with constant flag curvature.

\section{Zermelo deformations and the proof of \tref{perdeform} } \label{sec:3}

In this section, we discuss Zermelo deformations of Finsler metrics
with respect to a Killing vector field. Here we work in any dimension.

Let $F$ be a Finsler metric on $M^n$ and $X$ a Killing vector field of~$F$.
Under the Legendre transform $D\left(\frac{1}{2}F^2\right)\colon TM\to T^*M$,
we obtain the dual norm $F^*:T^*M \to \R$,
whose Hamiltonian flow with respect to the canonical symplectic structure
on $T^*M$ corresponds to the geodesic flow of~$F$ via the Legendre transform.

Let $\alpha\in\R$ be a constant.
The function $F^*_\alpha(\xi):=F^*(\xi)+\alpha\,\xi(X)$
defines a norm on~$T^*M$ if and only if $F(-\alpha X)<1$.
Applying to it the Legendre transform
$D\left(\frac{1}{2} (F^*_\alpha)^2\right)$
with respect to $\frac{1}{2}(F^*_\alpha)^2$,
we obtain a Finsler metric on $M$ which we denote by $F_\alpha$.

The Finsler metric~$F_\alpha$ is the \textit{Zermelo deformation}
of~$F$ with respect to the Killing field $\alpha X$.
It is easy to see that, for each $p$, its unit sphere in $T_pM$
is the unit sphere of~$F$ in~$T_pM$ translated by~$\alpha X(p)$.
Such metrics were examined in detail in~\cite{Ziller}
and can be viewed as generalizations of the examples of Katok~\cite{Katok},
in which~$F$ was a Riemannian metric of constant Gauss curvature on~$S^2$.

Denote by $\tau_t$ the flow of $X$
and by $d^*\tau_t$ its natural lift to $T^*M$;
it is generated by the Hamiltonian function
$\xi\in T^*M \mapsto \xi(X) \in \R$.
Denote by $\varphi_t$ the geodesic flow of $F$ on $T^*M$,
which we view as the flow of the Hamiltonian~$F^*$.
Then, the geodesic flow of~$F_\alpha$ is, by construction,
the flow of the Hamiltonian $F^*_\alpha$.
Since $\tau_t$ acts by isometries of $F$,
the flows $d^*\tau_t$ and $\varphi_t$ commute.
Hence, the flow of the Hamiltonian $F^*_\alpha=F^*(\xi)+\alpha \xi(X)$
is the composition $\varphi_t \circ d^*\tau_{\alpha t}$.
Thus, if $c(t)$ is a normal geodesic of~$F$,
then $c_\alpha(t)=\tau_{\alpha t}(c(t))$ is a normal geodesic of~$F_\alpha$.

In \cite{FM,HuangMo}, as already observed in~\cite{Foulon},
see also \cite{BRS},
it was shown that if $F$ has constant flag curvature~$1$,
then $F_\alpha$ does as well.
If $\psi$ is the isometry defined in~\sref{prthm1} for~$F$,
and $\psi_\alpha$ is its analog for $F_\alpha$,
then $\psi_\alpha=\psi\circ\tau_{\alpha\pi}$.

\begin{proposition}\label{perdeform2}
Let $F$ be a Finsler metric on a compact $n$-manifold~$M$
with constant flag curvature~$1$
that has a geodesic that is not closed.
Then there exists a Killing vector field~$X$
such that the Zermelo deformation~$F_1$
(i.e., $F^*_1(\xi)=F^*(\xi)+ \xi(X)$)
is a Finsler metric with constant flag curvature~$1$
and all of its geodesics are closed.
\end{proposition}

\begin{proof}
Since there exists a geodesic that is not closed,
the map~$\psi$ does not have finite order.
Thus $\{\psi^k\mid k\in\Z\}$ is an abelian subgroup
of the compact Lie group $\Isom(F)$ that is not closed.
Its closure is a compact abelian group
and hence has finitely many components,
with its identity component being a compact torus $T$ of positive dimension.
Thus some positive power of~$\psi$ lies in~$T$,
and, hence, for some large $k$ the isometry $\psi^k$
is close to the identity in~$\Isom(F)$.
Hence there exists a Killing vector field $X$
of small length and with flow $\tau_t$
such that $\tau_{-k\pi}=\psi^k$.
For the Finsler metric $\tilde F:= F_1$,
the corresponding map $\tilde \psi$
satisfies $\tilde \psi^k=\Id$ since, as explained above,
$\tilde \psi^k =\psi^k \circ\tau_{ k \pi}$.
\end{proof}

Notice that we can choose the vector field~$X$
such that it is arbitrary small,
which implies that the metric~$\tilde F$
can be chosen arbitrarily close to~$F$,
and that~\pref{perdeform2} reduces the classification
of Finsler metrics on~$S^n$ with constant flag curvature~$1$
to those for which all geodesics are closed.

In dimension $2$ we can say more:

\begin{proposition} \label{prop:4}
Let $F=F_\lambda$ be a Finsler metric on~$S^2$
with constant flag curvature~$1$ and rotation angle~$\lambda$,
that, in addition, admits a nontrivial Killing vector field~$X$.
Then, the Zermelo deformation
$F^*_\alpha(\xi)=F^*(\xi)+\alpha \xi(X)$ has the following properties:
\begin{enumerate}
 \item[$($a$)$] $F_\alpha$ is a Finsler metric
       for all $-2\lambda<\alpha<2{-}2\lambda$,
 \item[$($b$)$] The rotation angle for $F_\alpha$
       is $\lambda_\alpha=\lambda+\alpha/2$,
 \item[$($c$)$] There exists a unique $\alpha$ such that $F_\alpha$
       has all geodesics closed with the same prime length~$2\pi$.
\end{enumerate}
\end{proposition}
\begin{proof}
By Remark~\ref{rem:1} we may assume that,
up to multiplication by a constant,
$X=\frac{\partial }{\partial \phi}$
in the coordinates $(\theta, \phi)$ from~\sref{prthm1}.
Since in the coordinates~$(\theta, \phi)$,
the mapping $\psi$ is given by (\ref{eq:psi})
and  $X$ is proportional to $\frac{\partial }{\partial \phi}$ with a constant coefficient, the equator $\{\theta=0\}$
is the fixed point set of an isometry and is therefore a reversible geodesic. Depending on the orientation, we denote this geodesic by $\gamma^{\pm}$.

In order to prove (a) and (b),
consider the function on the cotangent bundle
given by $\xi\mapsto \xi(X)$.
Since $X$ is a Killing vector field,
it is an integral for the geodesic flow.
Then, the differential of this function
can be proportional to the differential of $F$
only at the points $(x,\xi)\in TM$ (assuming $\xi\ne0$)
such that $x$ lies  over  the  geodesics $\gamma^\pm$.
Indeed, otherwise the circle $\{\theta= \textrm{const}\ne 0\}$
is a geodesic, which is impossible
since the points $p$ and $\psi(p)$
lie on different sides of the equator~$\{\theta=0\}$.

Thus, the function $\xi\mapsto\xi(X)$
restricted to the unit cotangent bundle
has zero differential only at the points
corresponding to the geodesics $\gamma^{\pm}$.
By direct calculation, we see that,
at the points corresponding to $\gamma^+$, its value is $2{-}2\lambda$
and, at the points corresponding to $\gamma^-$, its value is $ 2\lambda$.
Here we use the convention  that the geodesic $\gamma^+$
is oriented in the $\frac{\partial}{\partial \phi}$-direction,
and $\gamma^-$ in the $-\frac{\partial}{\partial \phi}$-direction.
Thus, the function $\xi(X)$ has maximum $2{-}2\lambda$ and minimum $-2\lambda$. Hence~$F_\alpha$ is positive for all $-2\lambda<\alpha<2{-}2\lambda$
on nonzero vectors, which implies that $F_\alpha$ is a Finsler metric.

In order to prove (c), recall that $\psi_\alpha=\psi\circ \tau_{\alpha\pi}$,
which implies $\lambda_\alpha=\lambda+\alpha/2$.
Thus for part (c) we simply choose $\alpha=1{-}2\lambda$
and hence $\lambda_\alpha=\tfrac12$.
\end{proof}

Thus any Finsler metric $F_\lambda$ on $S^2$ with~$\lambda$ irrational
can be viewed as a Zermelo deformation of one for which all geodesics
are closed with the same length and which, in addition, admits a Killing vector field.  Meanwhile, we note that not all Finsler metrics on $S^2$
with constant flag curvature admit a nontrivial Killing field; in~\cite{Bryant1996,Bryant1997}, many examples were constructed
with all geodesics closed of length~$2\pi$ and finite isometry group.

Finally, observe that, for any $\lambda\in (0,\tfrac12]$,
the rotation angle~$\lambda_\alpha$ takes on each possible value
in $(0,\tfrac12]$ uniquely.
As already mentioned above, if $F$ is the round metric on~$S^2$,
then $F_\lambda$ is known as a Katok metric~\cite{Katok}.
This implies \cref{cor:1} in the Introduction.

\section{The proof of \tref{entropy}} \label{sec:5}

First observe that the $\pi$-shift along the geodesic flow
coincides with the diffeomorphism of $\mathrm{U}(M)$ induced by $\psi$,
which implies that it is an isometry of the natural metric on $\mathrm{U}(M)$
induced by the Binet-Legendre metric.
But any isometry of a compact metric space has zero topological entropy.

Next, we prove that the geodesic flow of~$F$,
viewed now as the Hamiltonian system on the cotangent bundle to~$M$
without the zero section, which we denote by~$T^*M\setminus\{0\}$,
and equipped with the canonical symplectic form and whose Hamiltonian~$F^*$
is the Legendre transform of~$F$ with respect to~$\tfrac12\,F^2$,
is Liouville integrable,
i.e., that there exist $n$ integrals that are functionally independent
almost everywhere and are in involution.

As in the proof of \pref{perdeform2},
we consider the closure of the group generated by $\psi$
in the group $\Isom(F)$.  Its $\Id$ component is a torus $T^k$
of dimension $k$, where we also allow $k= 0$.
The  action of the torus~$T^k$ on $M$
induces a Hamiltonian action of $T^k$ on the cotangent space $T^*M$.
The generators of this action, which we denote by $\xi_1,...,\xi_k$,
are the vector fields that generate the action of $T^k$ on $M$,
viewed as functions on $T^*M$.
We assume that the Hamiltonian flows of~$\xi_i$
are periodic with minimal period $1$.

Because of compactness of the isometry group of $M$,
there exists $\ell\in\mathbb{N}$
such that $\psi^\ell$ lies in the torus $T^k$.
This implies that the Arnold-Poisson action of $\R^{k+1}$
on $T^*M\setminus\{0\} $ generated
by the $k+1$ functions $ \xi_1,...,\xi_k, F^*$,
is a Hamiltonian action of a $k+1$-dimensional torus $T^{k+1}$.
The generators of this action are the functions $\xi_1,...,\xi_k$
and an appropriate linear combination of $F^*$ and $\xi_1,...,\xi_k$
constructed arguing as in~\sref{sec:3},
see the proof of \pref{perdeform2} there.
Note that $F^*$ is functionally independent from $\xi_1,..., \xi_k$.
Indeed, observe that in the proof of \pref{perdeform2},
we have shown that there exists a linear combination
$\xi= c_1 \xi_1+...+c_k \xi_k$ with constant coefficients
such that the flows  of the Hamiltonian systems
generated by $\xi$ and by $F^*$ satisfy
$\phi^\xi_{\pi}= \phi^{F^*}_{ \ell \pi }$.
If $F^*$ were functionally dependent on $\xi_1,..., \xi_k$
in some open subset, then the above observation
implies that, on this open set, the Hamiltonian vector field of $F^*$
is proportional to a  linear combination of $\xi_1,..., \xi_k$
with constant coefficients, which is clearly impossible.

We denote by $\textrm{Reg}$ the open dense set of regular points of the action, which is defined as the set of the points of $T^*M\setminus\{0\}$
where the differentials of the functions $\xi_1,...,\xi_k, F^*$
are linearly independent.  Next, consider the symplectic reduction
of $\textrm{Reg}$ with respect to $T^{k+1}$,
which is a smooth symplectic manifold of dimension $2n{-}2k{-}2$.
It is known that one can choose $n{-}k{-}1$ integrals
on any symplectic manifold of dimension $2n{-}2k{-}2$,
such that they are in involution and are functionally independent
almost everywhere, see~\cite[p.145]{Fomenko} or~\cite{BB}.
The pullback of these functions to $\textrm{Reg}$
gives us $n{-}k{-}1$ functions $I_{k+1},...,I_{n-1}$ on $\textrm{Reg}$,
such that $\xi_1,...,\xi_k, I_{k+1},...,I_{n-1}, F^*$
are in involution and functionally independent almost everywhere.
Furthermore, examining the proofs of~\cite[p.145]{Fomenko} or~\cite{BB},
we see that these functions can be extended to smooth functions
on all of $T^*M\setminus\{0\}$.
Hence the geodesic flow of $F$ is Liouville-integrable.

\section{The proof of \tref{Frankel} }  \label{example}

We first discuss a generalization of \tref{Frankel}
for Finsler metrics that are not necessarily reversible.
For this, we make the following definition.
Let~$\gamma_1$ and $\gamma_2$ be embedded closed geodesics
on~$S^2$ that do not intersect.
Then there exists a cylinder $S^1\times[0,1]$
whose boundary is the union of $\gamma_1$ and $\gamma_2$.
We say that the two closed curves \textit{have the same orientation}
if an orientation on~$C$ induces opposite orientations
on the boundary components~$\gamma_1$ and~$\gamma_2$.

\begin{proposition}\label{perdeform1}
Let $F$ be a Finsler metric on~$S^2$ with positive flag curvature.
If $\gamma_1$ and $\gamma_2$ are two closed geodesics
that are geodesically reversible,
or if they are embedded and have the same orientation,
then they intersect.
\end{proposition}

\begin{proof}
We first discuss the case when the closed geodesics $\gamma_i$
are embedded with the same orientation.
Let $\mathfrak{A}$ be the set of embedded closed curves in~$C$
that have the same orientation as $\gamma_i$.
By Arzela-Ascoli, there exists an embedded closed curve~$\alpha$
that realizes the minimum length among all curves in~$\mathfrak{A}$;
we assume that $\alpha$ is arc-length parametrized.

Let us show that either $\alpha$ coincides with one of the boundary geodesics
or is itself a geodesic that lies in the interior of $C$.
In order to do this, let us first observe
that its parts lying inside the interior of $C$
is a union of nonintersecting  geodesic segments.
Further observe that, at the points lying on the boundary of $C$,
the curve $\alpha$ has a velocity vector,
and this velocity vector is  tangent to the boundary,
i.e., proportional to the velocity vector of the corresponding $\gamma_i$
at this point.
Indeed, were $\alpha$ to have a `corner' at this point,
one could make it shorter by `shortcutting the corner'
by a geodesic within $C$, which would contradict its minimality.

Suppose $\alpha$ is tangent at a point to one of the boundary geodesics,
say, to the geodesic $\gamma_1$.
Then $\alpha$ cannot be tangent to $\gamma_1$ in such a way
that the velocity vectors are proportional at the point of tangency
with a negative  coefficient, since this would imply the existence of a self intersection of $\alpha$, which contradicts the minimality.
Thus, the velocity vectors of $\gamma_1$ and $\alpha$
are proportional at the point of tangency with a positive coefficient.
Then $\alpha$ coincides with $\gamma_1$
because it is globally and therefore locally minimal.

Thus $\alpha$ is either a geodesic in the interior of $C$
or agrees, including the orientation, with one of the boundary geodesics.
Now, we can repeat the proof of Synge's theorem:
There exists a smooth vector field along~$\alpha$,
pointing into the interior of~$C$ if~$\alpha$ is a boundary geodesic,
and parallel in the Chern connection.
From the second variation formula,
it follows that there exist shorter curves in $\mathfrak{A}$,
which is a contradiction.

The same proof works in the case where $\gamma_i$
are geodesically reversible, but not necessarily embedded.
Indeed, in that case, there exists a cylinder $C$
such that its boundary forms two closed piecewise smooth curves
with each segment a geodesic segment of one of $\gamma_i$.
Furthermore, at each vertex the two adjacent segments
are convex towards the interior of $C$.
We can then repeat the same argument as above.
The first statement of \tref{Frankel} is now proved.

We now prove the second statement of \tref{Frankel}
by constructing a Finsler metric on $S^2$
with two non intersecting closed geodesics.

Take a Katok metric $F=F_\lambda$ such that $\lambda$ is irrational.
In the spherical coordinates $(\theta,\phi )$ constructed in~\sref{prthm1},
the closed geodesics $\gamma^\pm$ lie on the equator $\{\theta=0\}$
and are, in fact, the only closed geodesics.
Though the image of these geodesics coincide,
the corresponding trajectories on the cotangent  bundles,
which we denote by $\bar\gamma^\pm\colon\R \to T^*S^2$, are disjoint.
They are trajectories of the geodesic flow
of the Hamiltonian $F^*\colon T^* S^2\to\R$,
and we may assume that they lie on the energy level $E:=(F^*)^{-1}(1)$.

Let $p_\theta$, $p_\phi$ be the canonical momenta
corresponding to the coordinates $\theta, \phi$.
Consider a  function $H:T^*M \to \R$ such that,
in a small neighborhood of $\bar\gamma^+$,
it coincides with the function $p_\theta$,
and, in a small neighborhood of $\bar\gamma^-$,
it coincides with the function $-p_\theta$ and is smooth.
Denote by $\Phi_t$ the Hamiltonian flow of the function $H$.
Near the curves $\bar\gamma^\pm$, the Hamiltonian vector field $X_H$
has coordinates
$$
\left(\frac{\partial H}{\partial p_\theta},
\frac{\partial H}{\partial p_\phi},-\frac{\partial H}{\partial \theta} , -\frac{\partial H}{\partial \phi}\right)= (\pm 1,0,0,0),
$$
which implies that, for small $t$ and near the points of $\gamma^\pm$,
the flow is given by $\Phi_t(\theta,\phi, p_\theta, p_\phi)
= (\theta \pm  t,\phi, p_\theta, p_\phi)$.
This means that the  the projections of the curves $\Phi_t(\bar\gamma^\pm)$
to the sphere are the curves $\{( \theta,\phi)\mid \theta=\pm t\}$.

Let $F^*_t=F^*\circ \Phi_{t}$ be a new Hamiltonian function
for a fixed small value of $t$.
Since $\Phi_t$ preserves the symplectic form,
the orbits of the Hamiltonian system generated by $F^*_t$
are $\Phi_t$ applied to the orbits of the Hamiltonian flow of $F^*$.
In particular, it has precisely two closed orbits,
which are the images of $\bar \gamma^\pm$ under $\Phi_t$.
Next, we define $\bar F_t\colon T^* S^2\to \R$
by the condition that it agrees with $F^*_t$ on $\Phi_{-t}(E)=(F^*_t)^{-1}(1)$
and is extended to be homogeneous of degree $1$.
For small $t$, $\bar F^*_t$ is smooth
on the cotangent bundle minus the zero section
and strictly convex, so it defines a Finsler metric on~$S^2$
with flag curvature close to $1$ and therefore positive.

The orbits of~$\bar F^*_t$ lying on~$\Phi_{-t}(E)$
are, up to a reparametrization,
the orbits of the flow generated by $ F^*_t$,
since the energy level $1$ of $\bar F^*_t$ and of $F^*_t$ coincide.
Thus all the geodesics of $\bar F_t$ other than two are not closed,
and these two are projected to the disjoint closed curves
$\{(\phi,\theta )\mid \theta= \pm t\}$.
\end{proof}

\begin{remark}
Notice that in the above example we can choose $H$
such that the Finsler metric~$\bar F_t$
has the same isometries as $F$, and hence admits a Killing vector field,
or such that its isometry group is finite.
\end{remark}

Let us shortly explain what the difficulty is
in higher dimensions of proving an analogue of Frankel's theorem.
For any Finsler metric $F$ and any $v\in T_pM$,
we have the Riemannian inner product $g^v(x,y)= D^2_\xi(F^2)_v(x,y)$ on $T_pM$,
where $D_\xi$ denotes the fiber derivative.
We furthermore have the Chern connection $D^v$
associated to $F$, that depends on a choice of $v$ as well.
If $\gamma$ is a geodesic of $F$ and $X$ is a vector field along $\gamma$,
then we have the second variation formula
$$
E''(0)=\int_a^b g^{v} (D^v_vX,D^v_vX)-g^v(R^v(X,v)v,X) dt +g^v(D^v_XX,v)|_a^b,
$$
where $v=\dot\gamma$. In the proof of Frankel's theorem,
one constructs a parallel vector field $X$
along a minimal geodesic connecting
the disjoint totally geodesic submanifolds~$N_1$ and~$N_2$
that is normal to them at the endpoints.
This defines a $1$-parameter variation
$c(t,s)=\exp_{\gamma(t)}(sX(t))$ of $\gamma$
connecting $N_1$ and $N_2$.
One would then like to use the second variation formula
to conclude that $E''(0)<0$ to obtain a contradiction.
Although the integral is negative, in the Finsler case,
the boundary term does not necessarily vanish,
as it does in the Riemannian case,
since $\eta(s)=c(a,s)$ is a geodesic,
which implies that $\frac{D^w \dot\eta}{dt}=0$ for $w=\dot\eta$,
but we need $\frac{D^v \dot\eta}{dt}=0$ for $v=\dot\gamma$ instead.

Note that recently a similar example appeared in~\cite{Rademacher}.
The metric constructed there
also has two closed, simple non-intersecting geodesics,
but may also have other closed geodesics.

\bigskip

R. L. Bryant:\
Department of Mathematics,
Duke University,
P.O. Box 90320,
Durham, NC 27708-0320,
USA.
\hspace{5pt}e-mail: bryant@math.duke.edu

P. Foulon:\
Centre International de Rencontres Math\'ematiques-CIRM,
163 avenue de Luminy,
Case 916, F-13288 Marseille - Cedex 9, France
\hspace{5pt}e-mail:  foulon@cirm-math.fr

S. Ivanov:\
St. Petersburg Department of Steklov Mathematical Institute,
Russian Academy
of Sciences, Fontanka 27, St.Petersburg 191023, Russia.
\hspace{5pt}e-mail: svivanov@pdmi.ras.ru

V. S. Matveev:\
Institut f\"ur Mathematik,
Fakult\"at f\"ur Mathematik und Informatik,\newline
Friedrich-Schiller-Universit\"at Jena,
07737 Jena, Germany.
\hspace{5pt}e-mail: vladimir.matveev@uni-jena.de

W. Ziller:\
Department of Mathematics,
University of Pennsylvania,
Philadelphia, PA 19104-6395, USA.
\hspace{5pt}e-mail: wziller@math.upenn.edu

\end{document}